\documentclass[a4paper, 12pt]{article}
\usepackage{amssymb}
\usepackage[english]{babel}
\usepackage{epsfig}
\usepackage{graphicx}
\usepackage{amsmath}
\usepackage{pdfsync}
\usepackage{scrextend}
\usepackage{subfig} 
\usepackage{epsfig} 

\large\normalsize \hbadness3000 \vbadness30000
\def\mytitle#1{\setcounter{equation}{0}
\setcounter{footnote}{0}
\begin{center}\Large\textbf{#1}\end{center}
\vspace{0.25cm}}
\def\myname#1{\centerline{{\large #1}}\vspace{-0.13cm}}

\newtheorem{theorem}{Theorem}[section]

\newtheorem{conjecture}[theorem]{Conjecture}
\newtheorem{corollary}[theorem]{Corollary}

\newtheorem{definition}[theorem]{Definition}
\newtheorem{example}[theorem]{Example}

\newenvironment{proof}[1][Proof]{\noindent\textbf{#1: }}{\hspace{\stretch{1}}\rule{0.5em}{0.5em}}

\begin{document}

\mytitle{On Erdos-Faber-Lovasz Conjecture }
\myname{S. M. Hegde, Suresh Dara}

\begin{center}
Department of Mathematical and Computational Sciences, \\National Institute of Technology Karnataka, \\Surathkal, Mangalore-575025
\\
{\em E-mail:} smhegde@nitk.ac.in, suresh.dara@gmail.com
\end{center}


\begin{abstract}
In 1972, Erd\"{o}s - Faber - Lov\'{a}sz (EFL) conjectured that, if $\textbf{H}$ is a linear hypergraph consisting of $n$ edges of cardinality $n$, then it is possible to color the vertices with $n$ colors so that no two vertices with the same color are in the same edge. In 1978, Deza, Erd\"{o}s and  Frankl had given an equivalent version of the same for graphs: Let $G= \bigcup _{i=1}^{n} A_i$ denote a graph with $n$ complete graphs $A_1, A_2,$ $ \dots , A_n$, each having exactly $n$ vertices and have the property that every pair of complete graphs has at most one common vertex, then the chromatic number of $G$ is $n$. 

The clique degree $d^K(v)$ of a vertex $v$ in $G$ is given by $d^K(v) = |\{A_i: v \in V(A_i), 1 \leq i \leq n\}|$. 
In this paper we give a method for assigning colors to
the graphs satisfying the hypothesis of the Erdös - Faber - Lovász conjecture using intersection matrix of the cliques $A_i$'s of $G$ and clique degrees of the vertices of $G$. Also, we
give theoretical proof of the conjecture for some class of graphs. In particular we show that:

\begin{enumerate}
\item If $G$ is a graph satisfying the hypothesis of the Conjecture \ref{EFL} and every $A_i$ ($1 \leq i \leq n$) has at most $\sqrt{n}$ vertices of clique degree greater than 1, then $G$ is $n$-colorable.

\item If $G$ is a graph satisfying the hypothesis of the Conjecture \ref{EFL} and every $A_i$ ($1 \leq i \leq n$) has at most $\left \lceil {\frac{n+d-1}{d}} \right \rceil$ vertices of clique degree greater than or equal to $d$ ($2\leq d \leq n$), then $G$ is $n$-colorable.
\end{enumerate}

\end{abstract}





\section{Introduction}

One of the famous conjectures in graph theory is Erd\"{o}s - Faber - Lov\'{a}sz conjecture. It states that if $\textbf{H}$ is a linear hypergraph consisting of $n$ edges of cardinality $n$, then it is possible to color the vertices of $\textbf{H}$ with $n$ colors so that no two vertices with the same color are in the same edge \cite{berge1990onvizing}. Erd\"{o}s, in 1975, offered 50 USD \cite{erdHos1975problems, arroyo2008dense} and in 1981, offered 500 USD \cite{erdHos1981combinatorial, jensen2011graph} for the proof or disproof of the conjecture. Kahn \cite{kahn1992coloring} showed that the chromatic number of $\textbf{H}$ is at most $n+o(n)$. Jakson et al.  \cite{jackson2007note} proved that the conjecture is true when the partial hypergraph $S$ of $\textbf{H}$ determined by the edges of size at least three can be $\Delta_S$-edge-colored and satisfies $\Delta_S \leq 3$. In particular, the conjecture holds when $S$ is unimodular and $\Delta_S \leq 3$. Sanhez - Arrayo \cite{arroyo2008dense} proved the conjecture for dense hypergraphs. Faber \cite{faber2010uniformregular} proves that for fixed degree, there can be only finitely many counterexamples to EFL on this class (both regular and uniform) of hypergraphs.  

\begin{conjecture}
 If $\textbf{H}$ is a linear hypergraph consisting of $n$ edges of cardinality $n$, then it is possible to color the vertices of $\textbf{H}$ with $n$ colors so that no two vertices with the same color are in the same edge.
\end{conjecture}

We consider the equivalent version of the conjecture for graphs given by Deza, Erd\"{o}s and  Frankl in 1978 \cite{deza1978intersection, arroyo2008dense, jensen2011graph, mitchem2010arscombin}.

\begin{conjecture}\label{EFL}
Let $G= \bigcup _{i=1}^{n} A_i$ denote a graph with $n$ complete graphs $(A_1, A_2,$ $ \dots , A_n)$, each having exactly $n$ vertices and have the property that every pair of complete graphs has at most one common vertex, then the chromatic number of $G$ is $n$.
\end{conjecture}

\begin{definition}\label{d1}
Let $G= \bigcup _{i=1}^{n} A_i$ denote a graph with $n$ complete graphs $A_1, A_2,$ $ \dots , A_n$, each having exactly $n$ vertices and the property that every pair of complete graphs has at most one common vertex. The clique degree $d^K(G)$ of a vertex $v$ in $G$ is given by $d^K(v) = |\{A_i: v \in V(A_i), 1 \leq i \leq n\}|$. The maximum clique degree $\Delta ^K(G)$ of the graph $G$ is given by $\Delta ^K(G)=max_{v\in V(G)}d^K(v)$.
\end{definition}

From the above definition, one can observe that degree of a vertex in hypergraph is same as the clique degree of a vertex in a graph.

\section{Coloring of $G$}

Let $G$ be the graph satisfying the hypothesis of Conjecture \ref{EFL}. Let $\hat{H}$ be the graph obtained by removing the vertices of clique degree one from graph $G$. i.e. $\hat{H}$ is the induced subgraph of $G$ having all the common vertices of $G$.

\begin{theorem}\label{SY1}
If $G$ is a graph satisfying the hypothesis of the Conjecture \ref{EFL} and every $A_i$ ($1 \leq i \leq n$) has at most $\sqrt{n}$ vertices of clique degree greater than 1, then $G$ is $n$-colorable.
\end{theorem}

\begin{proof}
Let $G$ be a graph satisfying the hypothesis of the Conjecture \ref{EFL} and every $A_i$ ($1 \leq i \leq n$) has at most $\sqrt{n}$ vertices of clique degree greater than 1. Let $\hat{H}$ be the induced subgraph of $G$ consisting of the vertices of clique degree greater than one in $G$. Define
$X=\{b_{i,j}:A_i \cap A_j = \emptyset \}$, $X_i=\{v\in G: d^K(v)=i\}$ for $i=1, 2, \dots, n$.

From \cite{arroyo2008dense}, it is true that the vertices of clique degree greater than or equal to $\sqrt{n}$ can be assigned with at most $n$ colors. Assign the colors to the vertices of clique degree in non increasing order. Assume we next color a vertex $v$ of clique degree $1<k<\sqrt{n}$. At this point only vertices of clique degree $\geq k$ have been assigned colors. By assumption every $A_i$ ($1 \leq i \leq n$) has at most $\sqrt{n}$ vertices of clique degree greater than 1 and clique degree of $v$ is $k$ ($k < \sqrt{n}$), then for these $k$ $A_i's$ there are at most $k \sqrt{n}<n$ vertices have been colored. Therefore, there is an unused color from the set of $n$ colors, then that color can be assigned to the vertex $v$.
\end{proof}

\begin{theorem}\label{SY2}
If $G$ is a graph satisfying the hypothesis of the Conjecture \ref{EFL} and every $A_i$ ($1 \leq i \leq n$) has at most $\left \lceil {\frac{n+d-1}{d}} \right \rceil$ vertices of clique degree greater than or equal to $d$ ($2\leq d \leq n$), then $G$ is $n$-colorable.
\end{theorem}

\begin{proof}
Let $G$ be a graph satisfying the hypothesis of the Conjecture \ref{EFL} and every $A_i$ ($1 \leq i \leq n$) has at most $\left \lceil {\frac{n}{d}} \right \rceil$ vertices of clique degree greater than or equal to $d$ ($2\leq d \leq n$). Let $\hat{H}$ be the induced subgraph of $G$ consisting of the vertices of clique degree greater than one in $G$. Define
$X=\{b_{i,j}:A_i \cap A_j = \emptyset \}$, $X_i=\{v\in G: d^K(v)=i\}$ for $i=1, 2, \dots, m$.

Assign the colors to the vertices of clique degree in non increasing order. Assume we next color a vertex $v$ of clique degree $k$. At this point only vertices of clique degree $\geq k$ have been assigned colors. By assumption every $A_i$ ($1 \leq i \leq n$) has at most $\sqrt{n}$ vertices of clique degree greater than 1 and clique degree of $v$ is $k$ ($k < \sqrt{n}$), then for these $k$ $A_i$'s there are at most $k (\left \lceil {\frac{n}{k}} \right \rceil -1)<n$ vertices have been colored. Therefore, there is an unused color from the set of $n$ colors, then that color can be assigned to the vertex $v$.
\end{proof}

Let $G$ be a graph satisfying the hypothesis of the Conjecture \ref{EFL}, the intersection matrix is a square $n \times n$ matrix $C$ such that its element $c_{i,j}$ is $c$ when $A_i \cap A_j \neq \emptyset$ otherwise zero for $i \neq j$ and the diagonal elements of the matrix are all zero.

Given below is a method to color graph $G$ satisfying the hypothesis of the Conjecture \ref{EFL} using intersection matrix (color matrix) of the cliques $A_i$'s of $G$ and clique degrees of the vertices of $G$.

\textbf{Method for assigning colors to graph $G$:}

Let $G$ be a graph satisfying the hypothesis of Conjecture \ref{EFL}. Let $\hat{H}$ be the induced subgraph of $G$ consisting of the vertices of clique degree greater than $1$ in $G$. Relable the vertex $v$ of clique degree greater than $1$ in $G$ by $u_x$, where $x= k_1, k_2, \dots, k_j$; vertex $v$ is in $A_{k_i}, 1 \leq i \leq j$. 
Define
$X=\{b_{i,j}:A_i \cap A_j = \emptyset \}$, $X_i=\{v\in G: d^K(v)=i\}$ for  $i=1, 2, \dots, n$.

Let $C$ be the intersection matrix of the cliques $A_i$'s of $G$ where $c_{i,j}= 0$ if $A_i \cap A_j =\emptyset$ otherwise $c$ for $i\neq j$ and $c_{i,i}$ is $0$. Let $1, 2, \dots, n$ be the $n$-colors.


Let $T_i=X_i$, $P_i= \emptyset$ and $S=\{j : T_j \neq \emptyset, 2\leq j \leq n\}$.

If $S=\emptyset$, then the graph $G$ has no vertex of clique degree greater than one, which implies $G$ has exactly $n^2$(maximum number) vertices. i.e., $G$ is $n$ components of $K_n$. Otherwise, consider the intersection matrix $C$ as the color matrix and follow the steps.


\begin{description}

    \item[Step 1:]
If $S = \emptyset$, stop the process. Otherwise, let $\max(S) = k$, for some $k, 2\leq k \leq n$. Then consider the sets $T_k$ and $P_k$, go to step 2.

     \item[Step 2:]
If $T_k=\emptyset$, go to step 1. Otherwise, choose a vertex $u_{i_1,i_2, \dots, i_k}$ from $T_k$, where $i_1 <i_2 < \dots < i_k$ and go to Step 3.

 \item[Step 3:]
Let $Y_{i}=\{y:$ color $y$ appears atleast $k-1$ times in the $i^{th}$ row of the color matrix $\}$, $i=1,2,\dots,n$. 
If $|{\bigcup_{i=i_1}^{i_k}Y_i}|=n$,  let $B_T=\bigcup_{i=2}^{n}P_i$, $B_P=\emptyset$ and go to Step 4. Otherwise, construct a new color matrix $C_1$ by putting least $x$ in $c_{i,j}$, where $x \in \{1, 2, 3, \dots n \} \setminus \bigcup_{i=i_1}^{i_k}Y_i$, $i\neq j$, $i_1 \leq i, j  \leq i_k$. Then add the vertex $u_{i_1,i_2, \dots, i_k}$ to $P_k$ and remove it from $T_k$, go to Step 2.

 \item[Step 4:]
Choose a vertex $v$ from $B_T$ such that $v\in A_i$, for some $i$, $i_1 \leq i \leq i_k$. Let $B=\{i:v\in A_i, 1\leq i \leq n\}$ and go to Step 5.
 
  \item[Step 5:]
  Let $Y_{i}=\{y:$ color $y$ appears atleast $k-1$ times in the $i^{th}$ row of the color matrix$\}$, for every $i \in B$. If $|{\bigcup_{i\in B}Y_i}|=n$ add the vertex $v$ to $B_P$ and remove it from $B_T$, go to Step 4. Otherwise  construct a new color matrix $C_2$ by putting $x$ in $c_{i,j}$, where $x \in \{1, 2, 3, \dots n \} \setminus \bigcup_{i\in B}Y_i$, $i\neq j$, $i, j  \in B$. Go to Step 3.

\end{description}

Thus, we get the modified color matrix $C_M$.
Then, color the vertex $v$ of $\hat{H}$ by $c_{i,j}$ of $C_M$, whenever $v \in A_i \cap A_j$. Then, extend the coloring of $\hat{H}$ to $G$. Thus $G$ is $n$-colorable.


Following is an example illustrating the above method.


\begin{example}
Let $G$ be the graph shown in Figure \ref{n6G}.

Let $V(A_1) = \{v_1, v_2, v_3, v_4, v_5, v_6\}$, 
$V(A_2) = \{v_1, v_7, v_8, v_9, v_{10}, v_{11}\}$,\\
$V(A_3) = \{v_1, v_{12}, v_{13}, v_{14}, v_{15}, v_{16}\}$,
$V(A_4) = \{v_1, v_{17}, v_{18}, v_{19}, v_{20}, v_{21}\}$,\\
$V(A_5) = \{v_6,  v_{7}, v_{16}, v_{22}, v_{23}, v_{24}\}$,
$V(A_6) = \{v_9, v_{16}, v_{19}, v_{25}, v_{26}, v_{27}\}$. 

Relabel the vertices of clique degree greater than one in $G$ by $u_A$ where $A=\{i:v\in A_i \text{ for } 1 \leq i \leq 6 \}$. The labeled graph is shown in Figure \ref{n6LG}. Figure \ref{n6LGH} is the graph $\hat{H}$, where $\hat{H}$ is obtained by removing the vertices of clique degree 1 from $G$.

\begin{figure}[h!]
\centering
\subfloat[Graph $G$]{\label{n6G}\includegraphics[width=3in]{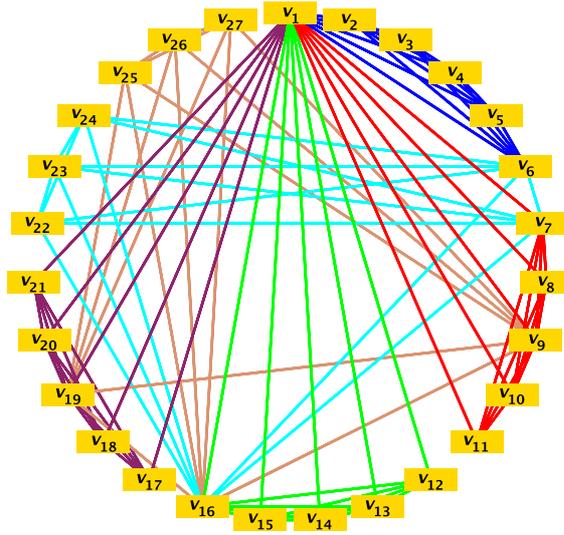}}\\
\subfloat[Graph $G$ after relabeling the vertices of clique degree greater than one]{\label{n6LG}\includegraphics[width=3in]{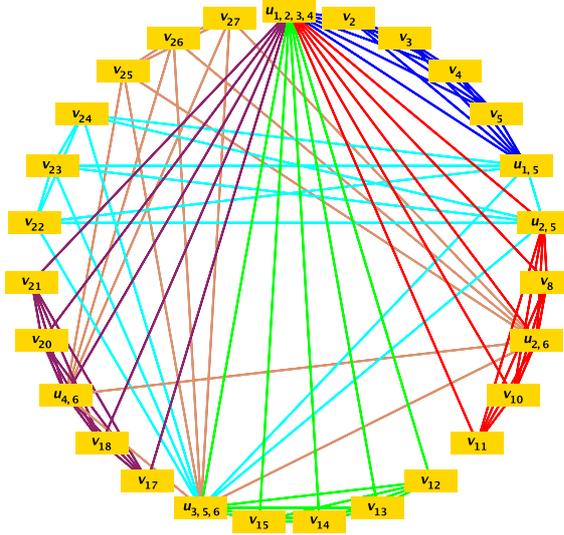}}
\caption{Graph $G$: before and after relabeling the vertices}
\label{}
\end{figure}

\begin{figure}[!h]
\centering
  \includegraphics[width=2.5in]{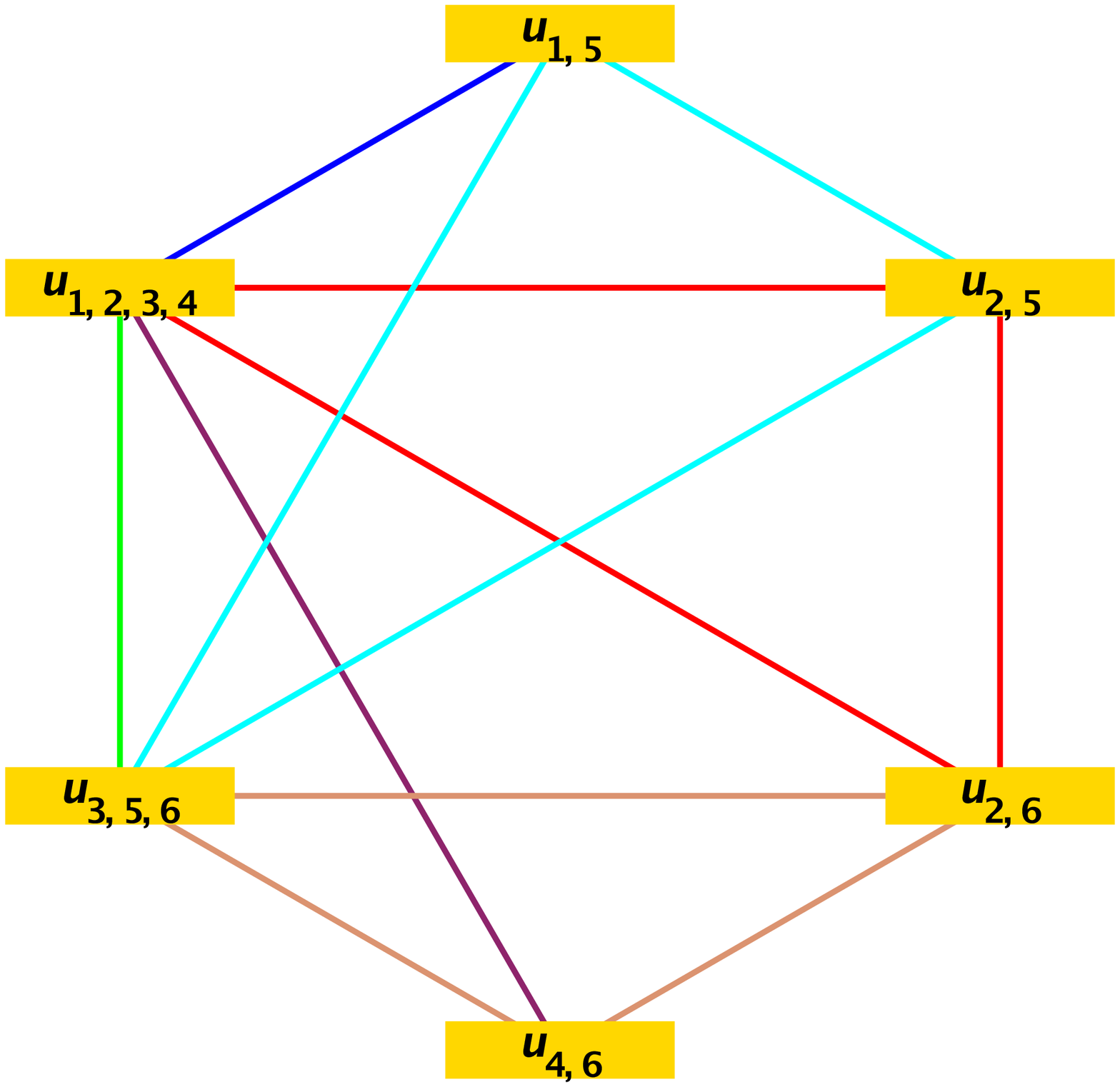}
  \caption{Graph $\hat{H}$}
\label{n6LGH}
\end{figure}

Let $X = \{b_{i,j}:A_i \cap A_j = \emptyset \} = \{b_{1,6}, b_{4,5}\} $,
\begin{eqnarray*}
X_1&=&\{v\in G: d^K(v)=1\} = \{v_2, v_3, v_5, v_{8}, v_{10}, v_{11}, v_{12}, v_{13}, v_{14}, v_{15}, \\ && v_{17}, v_{18}, v_{20}, v_{21}, v_{22}, v_{23}, v_{24}, v_{25}, v_{26}, v_{27}\},
\end{eqnarray*}
$X_2 = \{v\in G: d^K(v)=2\} = \{v_6, v_7, v_9, v_{19}\} = \{u_{1,5}, u_{2,5}, u_{2,6}, u_{4,6}\}$, \\
$X_3=\{v\in G: d^K(v)=3\} = \{v_{16}\} = \{u_{3,5,6}\} $,\\
$X_4=\{v\in G: d^K(v)=4\} = \{v_{1}\} = \{u_{1, 2, 3, 4}\} $, \\
$X_5 = \emptyset$ and $X_6=\emptyset$.

Let 1, 2, $\dots$, 6 be the six colors and $C = \left(\begin{array}{cccccc}
    0 & c & c & c & c & 0 \\
    c & 0 & c & c & c & c \\
    c & c & 0 & c & c & c \\
    c & c & c & 0 & 0 & c \\
    c & c & c & 0 & 0 & c \\
    0 & c & c & c & c & 0\\
  \end{array}\right)$ 

be the intersection matrix of order $6 \times 6$.

Consider the sets $T_i = X_i$, $P_i=\emptyset$ for $i =1,2,\dots 6$ and $S=\{j : T_j \neq \emptyset, 2\leq j \leq n\} = \{2,3,4\}$. Then by applying the method given above, we get the following.

\textbf{Step 1:}
Since $S \neq \emptyset$ and $\max (S) =4$, then choose the sets $T_4=  \{u_{1, 2, 3, 4}\}$ and $P_4 =\emptyset$. Go to step 2.

\textbf{Step 2:}
Since $T_4 \neq \emptyset$, choose the vertex $u_{1,2,3,4}$ from $T_4$, go to step 3.

\textbf{Step 3:}
Since $Y_1=\emptyset$, $Y_2=\emptyset$, $Y_3=\emptyset$, $Y_4=\emptyset$ and $|Y_1 \cup Y_2 \cup Y_3 \cup Y_4|<6$, choose the minimum color from the set $\{1,2,\dots, 6\}\setminus \cup_{i=1,2,3,4}Y_i$ and construct a new color matrix $C_1$ by putting $1$ in $c_{i,j}$, $i\neq j$, $i,j = 1,2,3,4$. Add the vertex $u_{1,2,3,4}$ to $P_4$ and remove it from $T_4$. Then

$C_1 = \left(\begin{array}{cccccc}
    0 & 1 & 1 & 1 & c & 0 \\
    1 & 0 & 1 & 1 & c & c \\
    1 & 1 & 0 & 1 & c & c \\
    1 & 1 & 1 & 0 & 0 & c \\
    c & c & c & 0 & 0 & c \\
    0 & c & c & c & c & 0\\
  \end{array}\right)$ ,

$T_4 = \emptyset$, $P_4 = \{u_{1,2,3,4}\}$. Go to step 2.

\textbf{Step 2:}
Since $T_4 = \emptyset$, go to step 1.

\textbf{Step 1:}
Since $S \neq \emptyset$ and $\max (S) =3$, then choose the sets $T_3=  \{u_{3,5,6}\}$ and $P_3 =\emptyset$. Go to step 2.

\textbf{Step 2:}
Since $T_3 \neq \emptyset$, choose the vertex $u_{3,5,6}$ from $T_3$, go to step 3.

\textbf{Step 3:}
Since $Y_3=\{1\}$, $Y_5=\emptyset$, $Y_6=\emptyset$, and $|Y_3 \cup Y_5 \cup Y_6|<6$, choose the minimum color from the set $\{1,2,\dots, 6\}\setminus \cup_{i=3,5,6}Y_i$ and construct a new color matrix $C_2$ by putting $2$ in $c_{i,j}$, $i\neq j$, $i,j = 3,5,6$. Add the vertex $u_{3,5,6}$ to $P_3$ and remove it from $T_3$. Then

$C_2 = \left(\begin{array}{cccccc}
    0 & 1 & 1 & 1 & c & 0 \\
    1 & 0 & 1 & 1 & c & c \\
    1 & 1 & 0 & 1 & 2 & 2 \\
    1 & 1 & 1 & 0 & 0 & c \\
    c & c & 2 & 0 & 0 & 2 \\
    0 & c & 2 & c & 2 & 0\\
  \end{array}\right)$,

$T_3 = \emptyset$, $P_3 = \{u_{3,5,6}\}$. Go to step 2.

\textbf{Step 2:}
Since $T_3 = \emptyset$, go to step 1.

\textbf{Step 1:}
Since $S \neq \emptyset$ and $\max (S) =2$, then choose the sets $T_2=  \{u_{1,5}, u_{2,5}, u_{2,6}, u_{4,6}\}$ and $P_2 =\emptyset$. Go to step 2.

\textbf{Step 2:}
Since $T_2 \neq \emptyset$, choose the vertex $u_{1,5}$ from $T_2$, go to step 3.

\textbf{Step 3:}
Since $Y_1=\{1\}$, $Y_5=\{2\}$ and $|Y_1 \cup Y_5|<6$, choose the minimum color from the set $\{1,2,\dots, 6\}\setminus \cup_{i=1,5}Y_i$ and construct a new color matrix $C_3$ by putting $3$ in $c_{i,j}$, $i\neq j$, $i,j = 1,5$. Add the vertex $u_{1,5}$ to $P_2$ and remove it from $T_2$. Then

$C_3 = \left(\begin{array}{cccccc}
    0 & 1 & 1 & 1 & 3 & 0 \\
    1 & 0 & 1 & 1 & c & c \\
    1 & 1 & 0 & 1 & 2 & 2 \\
    1 & 1 & 1 & 0 & 0 & c \\
    3 & c & 2 & 0 & 0 & 2 \\
    0 & c & 2 & c & 2 & 0\\
  \end{array}\right)$, 

$T_2=  \{u_{2,5}, u_{2,6}, u_{4,6}\}$, $P_2 = \{u_{1,5}\}$. Go to step 2.

\textbf{Step 2:}
Since $T_2 \neq \emptyset$, choose the vertex $u_{2,5}$ from $T_2$, go to step 3.

\textbf{Step 3:}
Since $Y_2=\{1\}$, $Y_5=\{2,3\}$ and $|Y_2 \cup Y_5|<6$, choose the minimum color from the set $\{1,2,\dots, 6\}\setminus \cup_{i=2,5}Y_i$ and construct a new color matrix $C_4$ by putting $4$ in $c_{i,j}$, $i\neq j$, $i,j = 2,5$. Add the vertex $u_{2,5}$ to $P_2$ and remove it from $T_2$. Then

$C_4 = \left(\begin{array}{cccccc}
    0 & 1 & 1 & 1 & 3 & 0 \\
    1 & 0 & 1 & 1 & 4 & c \\
    1 & 1 & 0 & 1 & 2 & 2 \\
    1 & 1 & 1 & 0 & 0 & c \\
    3 & 4 & 2 & 0 & 0 & 2 \\
    0 & c & 2 & c & 2 & 0\\
  \end{array}\right)$, 

$T_2=  \{u_{2,6}, u_{4,6}\}$, $P_2 = \{u_{1,5}, u_{2,5}\}$. Go to step 2.

\textbf{Step 2:}
Since $T_2 \neq \emptyset$, choose the vertex $u_{2,6}$ from $T_2$, go to step 3.

\textbf{Step 3:}
Since $Y_2=\{1,4\}$, $Y_6=\{2\}$ and $|Y_2 \cup Y_6|<6$, choose the minimum color from the set $\{1,2,\dots, 6\}\setminus \cup_{i=2,6}Y_i$ and construct a new color matrix $C_5$ by putting $3$ in $c_{i,j}$, $i\neq j$, $i,j = 2,6$. Add the vertex $u_{2,6}$ to $P_2$ and remove it from $T_2$. Then

$C_5 = \left(\begin{array}{cccccc}
    0 & 1 & 1 & 1 & 3 & 0 \\
    1 & 0 & 1 & 1 & 4 & 3 \\
    1 & 1 & 0 & 1 & 2 & 2 \\
    1 & 1 & 1 & 0 & 0 & c \\
    3 & 4 & 2 & 0 & 0 & 2 \\
    0 & 3 & 2 & c & 2 & 0\\
  \end{array}\right)$, 

$T_2=  \{u_{4,6}\}$, $P_2 = \{u_{1,5}, u_{2,5}, u_{2,6}\}$. Go to step 2.

\textbf{Step 2:}
Since $T_2 \neq \emptyset$, choose the vertex $u_{4,6}$ from $T_2$, go to step 3.

\textbf{Step 3:}
Since $Y_4=\{1\}$, $Y_6=\{2,3\}$ and $|Y_4 \cup Y_6|<6$, choose the minimum color from the set $\{1,2,\dots, 6\}\setminus \cup_{i=4,6}Y_i$ and construct a new color matrix $C_6$ by putting $4$ in $c_{i,j}$, $i\neq j$, $i,j = 4,6$. Add the vertex $u_{4,6}$ to $P_2$ and remove it from $T_2$. Then

$C_6 = \left(\begin{array}{cccccc}
    0 & 1 & 1 & 1 & 3 & 0 \\
    1 & 0 & 1 & 1 & 4 & 3 \\
    1 & 1 & 0 & 1 & 2 & 2 \\
    1 & 1 & 1 & 0 & 0 & 4 \\
    3 & 4 & 2 & 0 & 0 & 2 \\
    0 & 3 & 2 & 4 & 2 & 0\\
  \end{array}\right)$, 

$T_2=\emptyset$, $P_2 = \{u_{1,5}, u_{2,5}, u_{2,6}, u_{4,6}\}$. Go to step 2.

\textbf{Step 2:}
Since $T_2 = \emptyset$, go to step 1.

\textbf{Step 1:}
Since $S = \emptyset$, stop the process.

Assign the colors to the graph $\hat{H}$ using the matrix $C_M=C_6$, i.e., color the vertex $v$ by the $(i, j)$-th entry $c_{i,j}$ of the matrix $C_M$, whenever $A_i \cap A_j \neq \emptyset$ (see Figure \ref{n6CLGH}), where the numbers 1, 2, 3, 4, 5, 6 corresponds to the colors Maroon, Tan, Green, Red, Blue, Cyan respectively. Extend the coloring of $\hat{H}$ to $G$ by assigning the remaining colors which are not used for $A_i$ from the set of $6$-colors to the vertices of clique degree one in each $A_i$, $1 \leq i \leq 6$. The colored graph $G$ is shown in Figure \ref{n6CLG}.

\begin{figure}[h!]
\centering
\subfloat[Graph $\hat{H}$]{\label{n6CLGH}\includegraphics[width=2.5in]{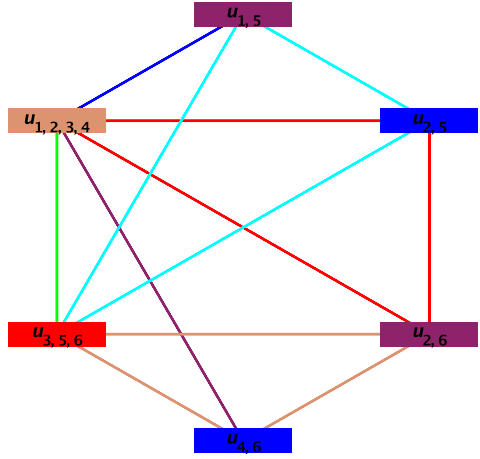}}\\
\subfloat[A 6 coloring of graph $G$]{\label{n6CLG}\includegraphics[width=3.5in]{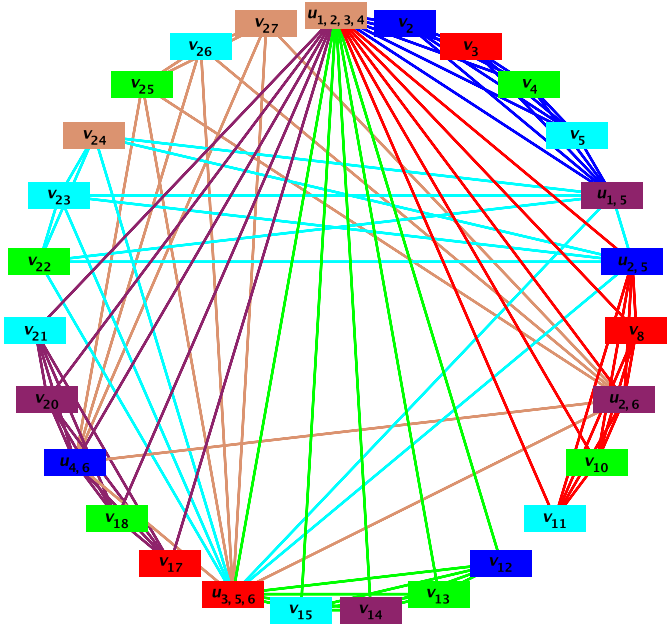}}
\caption{The graphs $\hat{H}$ and $G$, after colors have been assigned to their vertices.}
\label{}
\end{figure}

\begin{figure}[h!]
\centering
\includegraphics[width=3.5in]{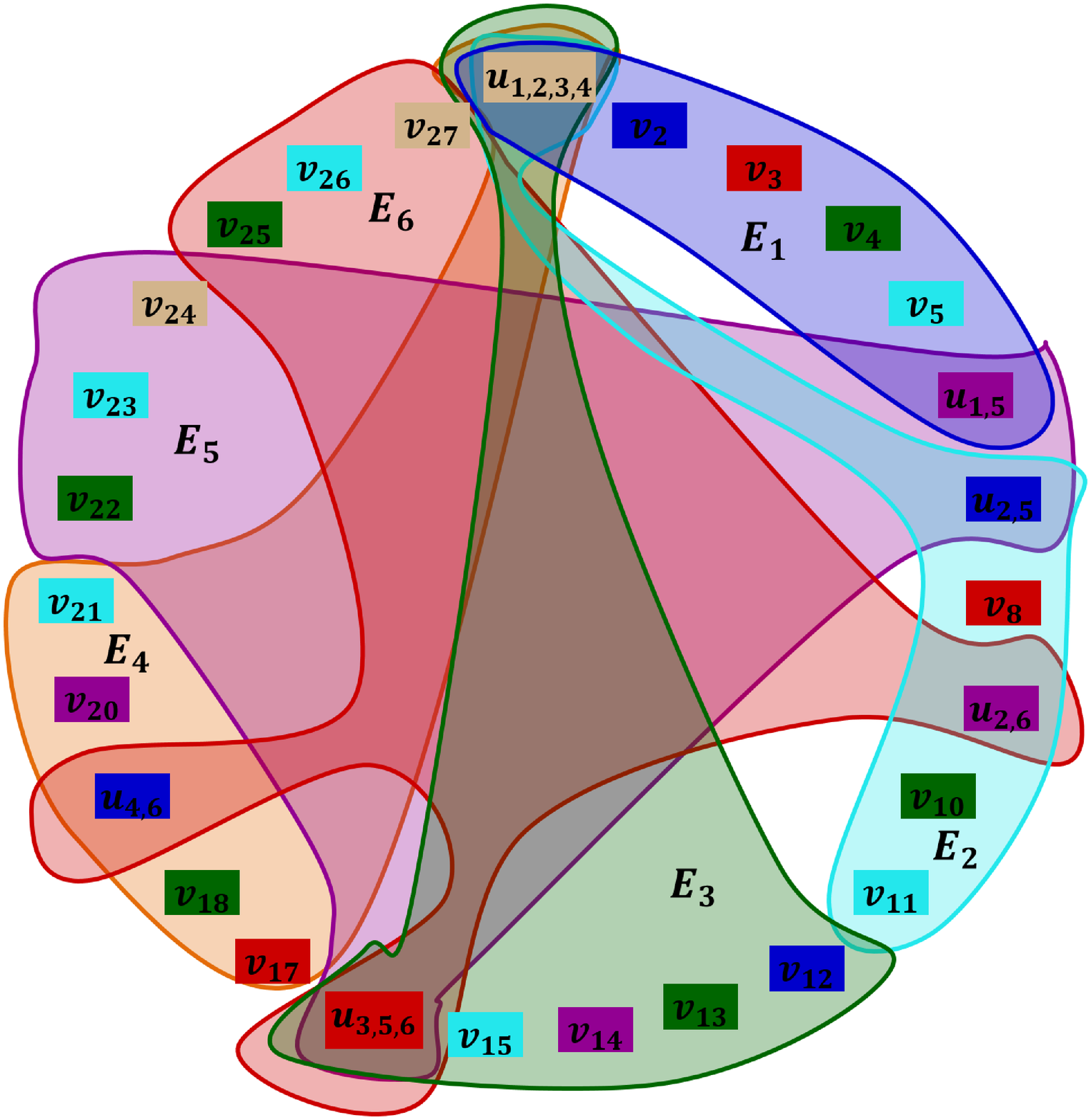}
\caption{A $6$ coloring of hypergraph {\bf{H}} corresponding to the graph $G$ shown in Figure \ref{n6CLG}}
\label{G1_CHY}
\end{figure}

\end{example}


The following results give a relation between the number of complete graphs and clique degrees of a graph.

\begin{theorem}\label{t2}
Let $G$ be a graph satisfying the hypothesis of Conjecture \ref{EFL}. Then if the intersection of any two $A_i$'s is non empty, then 

$$\binom{d^K(v_1)}{2} + \binom{d^K(v_2)}{2} + \dots + \binom{d^K(v_l)}{2} = \frac{n(n-1)}{2},$$

where $\{v_1, v_2, \dots, v_l\}$ is the set of all vertices of clique degree greater than $1$ in $G$.
\end{theorem}

\begin{proof}
If $G$ is isomorphic to the graph $H_n$ for some $n$, then the result is obvious. If not there exists at least one vertex $v$ of clique degree greater than $2$. Define $I_v=\{i:v\in A_i\}$ then $d^K(v) = |I_v|=p$. For every unordered pair of elements $(i, j)$ of $I_v$ there is a vertex $b_{ij} (\text{ where } i<j)$ in $H_n$. Therefore corresponding to the elements of $I_v$ there are $\binom{p}{2}$ vertices in $H_n$. Since $G$ satisfies the hypothesis of Conjecture \ref{EFL}, there is no vertex $v'$ different from $v$ in $G$ such that $v'\in A_i\cap A_j$ where $i,j \in I_v$. Therefore for every vertex $v$ of clique degree greater than $1$ in $G$, there are $\binom{d^K(v)}{2}$ vertices of clique degree greater than $1$ in $H_n$. As there are $\frac{n(n-1)}{2}$ vertices of clique degree greater than $1$ in $H_n$, $\frac{n(n-1)}{2} =\binom{d^K(v_1)}{2} + \binom{d^K(v_2)}{2} + \dots + \binom{d^K(v_l)}{2}$ where $\{v_1, v_2, \dots, v_l\}$ is the set of all vertices of clique degree greater than $1$ in $G$.
\end{proof}


\begin{corollary}\label{c1}
If $G$ is a graph satisfying the hypothesis of conjecture \ref{EFL}, then $G$ has at most $\frac{\binom{n}{2}}{\binom{m}{2}}$ vertices of clique degree $m$ where $m\geq 2$.
\end{corollary}

\begin{proof}
Let $A=\{v_1, v_2, \dots , v_l\}$ be the set of vertices of clique degree greater than $1$ in $G$ and $p=\frac{\binom{n}{2}}{\binom{m}{2}}$.
We have to prove that $G$ has at most $p$ vertices of clique degree $m$.
Suppose $G$ has $q>p$ vertices of clique degree $m$. Then by the definition of $A$, it follows that, $q$ vertices are in $A$. Let those vertices be $v_1, v_2, \dots, v_q$.
By Theorem \ref{t2} we get,

\begin{eqnarray*}
\frac{n(n-1)}{2}& =&\binom{d^K(v_1)}{2} + \binom{d^K(v_2)}{2} + \dots + \binom{d^K(v_l)}{2}\\
& \geq & \binom{d^K(v_1)}{2} + \binom{d^K(v_2)}{2} + \dots + \binom{d^K(v_q)}{2}\\
& =& q \binom{m}{2}\\
& \geq & (p+1)\binom{m}{2}\\
\frac{\binom{n}{2}}{\binom{m}{2}} & \geq & p+1\\
p & \geq & p+1,
\end{eqnarray*}

which is a contradiction. Hence there are at most $\frac{\binom{n}{2}}{\binom{m}{2}}$ vertices of clique degree $m$ in $G$, where $m\geq 2$.
\end{proof}


\bibliographystyle{plain}


\begin{thebibliography}{10}

\bibitem{berge1990onvizing}
Claude Berge.
\newblock On two conjectures to generalize {V}izing's theorem.
\newblock {\em Matematiche (Catania)}, 45(1):15--23 (1991), 1990.
\newblock Graphs, designs and combinatorial geometries (Catania, 1989).

\bibitem{deza1978intersection}
M~Deza, P~Erd{\"o}s, and P~Frankl.
\newblock Intersection properties of systems of finite sets.
\newblock {\em Proceedings of the London Mathematical Society}, 3(2):369--384,
  1978.

\bibitem{erdHos1981combinatorial}
P{\'a}l Erd{\H{o}}s.
\newblock On the combinatorial problems which i would most like to see solved.
\newblock {\em Combinatorica}, 1(1):25--42, 1981.

\bibitem{erdHos1975problems}
PAUL Erd{\H{o}}s.
\newblock Problems and results in graph theory and combinatorial analysis.
\newblock {\em Proc. British Combinatorial Conj., 5th}, pages 169--192, 1975.

\bibitem{faber2010uniformregular}
Vance Faber.
\newblock The {E}rd{\H o}s-{F}aber-{L}ov\'asz conjecture---the uniform regular
  case.
\newblock {\em J. Comb.}, 1(2):113--120, 2010.

\bibitem{jackson2007note}
Bill Jackson, G.~Sethuraman, and Carol Whitehead.
\newblock A note on the {E}rd{\H o}s-{F}arber-{L}ov\'asz conjecture.
\newblock {\em Discrete Math.}, 307(7-8):911--915, 2007.

\bibitem{jensen2011graph}
Tommy~R Jensen and Bjarne Toft.
\newblock {\em Graph coloring problems}, volume~39.
\newblock John Wiley \& Sons, 2011.

\bibitem{kahn1992coloring}
Jeff Kahn.
\newblock Coloring nearly-disjoint hypergraphs with $n+o(n)$ colors.
\newblock {\em Journal of Combinatorial Theory, Series A}, 59(1):31--39, 1992.

\bibitem{mitchem2010arscombin}
John Mitchem and Randolph~L. Schmidt.
\newblock On the {E}rd{\H o}s-{F}aber-{L}ov\'asz conjecture.
\newblock {\em Ars Combin.}, 97:497--505, 2010.

\bibitem{arroyo2008dense}
Abd{\'o}n S{\'a}nchez-Arroyo.
\newblock The {E}rd{\H o}s-{F}aber-{L}ov\'asz conjecture for dense hypergraphs.
\newblock {\em Discrete Math.}, 308(5-6):991--992, 2008.

\end{thebibliography}

\end{document}